\documentclass[11pt,a4paper]{article}
\usepackage[a4paper, total={5.3in, 7.25in}]{geometry}
\usepackage[utf8]{inputenc}
\usepackage{xcolor,amsmath}
\usepackage{amsthm}
\usepackage{amsfonts}
\usepackage{amssymb}
\usepackage{cite}
\usepackage{relsize}
\usepackage{blindtext, titlefoot}
\usepackage{sectsty}
\usepackage{hyperref}
\sectionfont{\fontsize{12}{15}\selectfont}
\newtheorem{lemma}{Lemma}
\newtheorem{theorem}{Theorem}

\title{ On the Value-Distribution of Hurwitz Zeta-Functions with Algebraic Parameter}
\author{Athanasios Sourmelidis $\cdot$ J\"orn Steuding}
\begin{document}
\date{}
\maketitle\unmarkedfntext{MSC 2010: 11M35}
\begin{abstract}
\noindent
We study the value-distribution of the Hurwitz zeta-function with algebraic irrational parameter $\zeta(s;\alpha)=\sum_{n\geq_0}(n+\alpha)^{-s}$.
In particular, we prove effective denseness results of the Hurwitz zeta-function and its derivatives in suitable strips containing the right boundary of the critical strip $1+i\mathbb{R}$. 
This may be considered as a first "weak`` manifestation of universality for those zeta-functions.
\end{abstract}
\section{Introduction and Statement of the Main Results}
Let $s=\sigma+it$ denote a complex variable. 
The Hurwitz zeta-function with a real parameter $\alpha\in(0,1]$ is for $\sigma>1$ defined by the Dirichlet series expansion
$$
\zeta(s;\alpha)=\sum_{n=0}^{\infty}(n+\alpha)^{-s},
$$
and by analytic continuation elsewhere except for a simple pole at $s=1$. 
This function was introduced by Hurwitz \cite{hurwitz1882einige} in 1881/2 and generalizes the famous Riemann zeta-function which appears as $\zeta(s)=\zeta(s;1)$. 

The Riemann zeta-function possesses a remarkable approximation property. 
In 1975, Voronin \cite{voronin1975theorem} proved that, roughly speaking, every non-vanishing analytic function $f$ defined on a sufficiently small disk $K$ centered at the origin can be approximated as good as we please by certain shifts of the Riemann zeta-function, 
\begin{align*}
\max\limits_{s\in K}\left\vert \zeta\left(s+\dfrac{3}{4}+i\tau\right)-f(s)\right\vert<\varepsilon\, ;
\end{align*}
furthermore, this approximation is a regular phenomenon: the set of shifts $\tau$ satisfying the above inequality has positive lower density. 
Since a single
 function approximates elements of a huge class of target functions $f$, this property is called {\it universality}. 

Voronin's celebrated universality theorem has been generalized and extended in various ways. 
It has been shown in 1979/81 by Gonek \cite{gonek1979analytic} and (independently) Bagchi \cite{bagchi1981statistical} that the Hurwitz zeta-function $\zeta(s;\alpha)$ satisfies the analogue of Voronin's universality theorem whenever $\alpha$ is rational or transcendental. 
It appears that for every such $\alpha\neq \frac{1}{2},1$, the target function $f$ even may vanish in $K$, for which those $\zeta(s;\alpha)$ are said to be {\it strongly} universal. 
For this and more details we refer besides the original works to \cite{steuding2007value}. 

Ever since the question whether the Hurwitz zeta-function with an algebraic irrational parameter is universal in this or another sense has been investigated, so far only with little success though. 
For instance, Garunk\v stis \cite{garunkvstis2005note} showed by employing the continuity of $\zeta(s;\alpha)$ with respect to $\alpha$ the existence of zeros in the right half of the critical strip except for $\alpha=\frac{1}{2},1$ (which would also follow from universality). 
Also, Laurin\v cikas and the second author \cite{laurinvcikas2005limit} obtained limit theorems for a Hurwitz zeta-function with an algebraic irrational parameter (unfortunately not sufficiently explicit for being used in a hypothetical proof of universality). 
Lastly, Mishou \cite{zbMATH05263981} considered the  joint value distribution of the Riemann zeta-function and the Hurwitz zeta-function with algebraic parameter on the right of the line $1+i\mathbb{R}$.

In this article we study the behaviour of the Hurwitz zeta-function $\zeta(s;\alpha)$ on the left of $1+i\mathbb{R}$, where the parameter $\alpha$ is an algebraic irrational number.
 We incorporate ideas of Voronin \cite{voronin1989omega} and Good \cite{good1980distribution} to obtain quantitative results; an additional feature in our reasoning is the use of the theory of approximation by algebraic numbers.

First, we have to introduce several notations which will be kept throughout the paper.
The naive height of a complex polynomial $P(X)$, denoted by $H(P)$, is the maximum of the absolute values of its coefficients.
 If $\alpha$ is an algebraic number, then its degree and height, which we denote by $d(\alpha)$ and $H(\alpha)$, are defined to be the degree and the height of its minimal polynomial over $\mathbb{Z}$, respectively.  
All the constants appearing in the sequel are effectively computable.
The numbers $R$, $Q$ and $M$ will always denote positive integers,  while $T$, $\sigma$ and $d$ will be positive real numbers. 
Lastly, we postpone the definitions for the number $\mathbf{E}=\mathbf{E}(R,Q,\sigma)$, the set $$\mathcal{A}(Q,M)\subseteq\mathbb{A}:=\left\{a\in[0,1]:\alpha\,\,\text{is algebraic irrational}\right\}$$ and the number $\mathbf{K}=\mathbf{K}(Q,M,\alpha)$ which appear in Theorem \ref{weak tran} until Section \ref{Auxil} (see \eqref{E_def}, \eqref{al} and \eqref{B}, respectively).
Our main result is the following:
\begin{theorem}\label{weak tran}
For every $\sigma\in(1/2,1]$, $N\in\mathbb{N}$, $A\in(0,1]$ and $d\geq3$, there exist positive numbers $c_0$, $c_1$, $c_2$  which depend on $\sigma$ and $N$, $c_3=c_3(N,A)$, $c_4$, $c_5=c_5(N,d)$ and $\nu=\nu(d,N)$, such that the following is true:

Let $\varepsilon>0$ and $\mathbf{a}:=(a_0,\dots,a_n)\in\mathbb{C}^{N+1}$.
Let also 
$$R\geq c_0\,\varepsilon^{4/(1-2\sigma)}$$
 and $Q_0\geq c_1 R$ be
 positive integers satisfying the system of inequalities
\begin{align*}
c_2\left(|a_k|+A^{-1/2}\right)
\leq\mathbf{E}\left(\dfrac{\log\left(\frac{Q_0}{R+1}\right)}{2N\log Q_0}\right)^{N}k!(N-k)!\left(\log Q_0\right)^k,\,\,\,k=0,\dots,N.
\end{align*}
Then, for any $Q\geq c_3\left(Q_0+1/\varepsilon^8\right)$, $M\geq c_4\exp\left(2Q^2\right)$, $\alpha\in\mathcal{A}(Q,M)\cap[A,1]$ of degree $d(\alpha)\leq d-1$, where
\begin{align}\label{degree}
d+\dfrac{1}{2d}
\leq\dfrac{40}{267}\left(\dfrac{1}{3(1-\sigma)}\right)^{1/2}
\end{align}
and the left-hand side of the inequality is $+\infty$ for $\sigma=1$,
and any
\begin{align}\label{BBB}
{T\geq c_5\max\left\{\left(\mathbf{K}\exp\left((M+2)\exp\left(Q^2\right)\right)\right)^{\frac{4d}{4(d-d(\alpha))-3}},\varepsilon^{-2\nu}\right\}},
\end{align}
there is $\tau\in[T,2T]$ with
\begin{align}\label{SYST}
\left|\zeta^{(k)}\left(\sigma+i\tau;\alpha\right)-a_k\right|
<\varepsilon,\,\,\,\,\,\,\,\,\,k=0,\dots ,N.
\end{align}
Moreover, if $\mathcal{M}_T(\alpha,\sigma)$ is the set of those $\tau\in[T,2T]$ for which
\begin{align*}
\left|\zeta^{(k)}\left(\sigma+i\tau;\alpha\right)-a_k\right|
<\left(2\dfrac{Q^2+1}{Q^2-1}\right)^{1/2}\varepsilon,\,\,\,\,\,\,\,\,\,k=0,\dots ,N,
\end{align*}
then
\begin{align*}
\liminf\limits_{T\to\infty}\dfrac{1}{T}\mathrm{m}\left(\mathcal{M}_T(\alpha,\sigma)\right)\geq\dfrac{1}{2
}Q^{-2Q}\left(1-Q^{-2}\right),
\end{align*}
where $\mathrm{m}$ denotes the Lebesgue measure.
\end{theorem}

Observe that the theorem has meaning only when $\sigma\geq1-\xi$, where
$$\xi:=\dfrac{2^8\cdot 5^2}{3\cdot19^2\cdot89^2}\approx0.000746,$$
as follows from \eqref{degree} for $d=3$.
 As a consequence we obtain an effective but weak form of universality in the same manner as in \cite{garunkvstis2010effective}:

\begin{theorem}\label{weak.}
Let $1-\xi\leq\sigma_0\leq1$, $s_0=\sigma_0+it_0$  and $f\,:\,\mathcal{K}\to\mathbb{C}$ be continuous and analytic in the interior of $\mathcal{K}=\{s\in\mathbb{C}\,:\,\vert s-s_0\vert\leq r\}$, where $r>0$.
 Let also $0<A<1$ and $\varepsilon\in(0,\vert f(s_0)\vert)$. Then, for all but finitely many algebraic 
irrationals $\alpha$
in $[A,1]$ of degree at most
$d_0-1$, where
\begin{align*}
d_0+\dfrac{1}{2d_0}=\dfrac{40}{267}\left(\dfrac{1}{3(1-\sigma_0)}\right)^{1/2},
\end{align*}
 there exist real numbers $\tau\in[T,2T]$ and $\delta=\delta(\varepsilon,f,T)>0$ such that
$$
\max_{|s-s_0\vert\leq \delta r}\vert \zeta(s+i\tau;\alpha)-f(s)\vert<3\varepsilon,
$$
whenever $T=T(\varepsilon,f,\alpha)$ satisfies \eqref{BBB}.
The set of the exceptional $\alpha$ can be described effectively, while the dependence of $T$ on $f$ arises from the first $N$ Taylor coefficients of $f$ for sufficiently large $N$.
\end{theorem}

The restriction on the strip of universality reminds us of the case of the Dedekind zeta-function $\zeta_K$, where $K$ is an algebraic number field over $\mathbb{Q}$. Reich \cite{zbMATH03590361}, \cite{zbMATH03670533} proved that $\zeta_K$ is universal in the sense of Voronin in the strip $\max\left\{1/2,1-1/d\right\}$, where $d=[K:\mathbb{Q}]$ is the degree of the number field.

In the following section we list several well-known results that will turn out useful for our proofs, in particular an approximate functional equation for $\zeta(s;\alpha)$ and a Liouville type inequality. 
In the succeeding two sections we provide the proofs of the results mentioned above and we conclude with a few remarks which might be of interest with respect to further studies of this topic.

\section{Preliminaries}

\begin{lemma}\label{Hilbert}
Let $x_1,\dots,x_n$ be elements of a complex Hilbert space $\mathcal{H}$ and let $a_1,\dots,a_n$ be complex numbers with $|a_j|\leq1$ for $1\leq j\leq n$.
Then there exist complex numbers $b_1,\dots,b_n$ with $|b_j|=1$ for $1\leq j\leq n$, satisfying the inequality
\begin{align*}
\left|\left|\mathlarger\sum\limits_{j=1}^{n}a_jx_j-\mathlarger\sum\limits_{j=1}^{n}b_jx_j\right|\right|^2\leq 4\mathlarger\sum\limits_{j=1}^{n}||x_j||^2
\end{align*}
\end{lemma}

\begin{proof}
For a proof see \cite[Lemma 5.2]{steuding2007value}.
\end{proof}

\begin{lemma}\label{Hil}
Let $X$ be a locally convex vector space. 
Let $K\subseteq X$ be a closed convex set, and suppose that $z\in X\setminus K$. 
Then there exists a continuous linear functional $\ell\in X^*$ and a constant $c\in\mathbb{R}$ such that $\ell(y)\leq c<\ell(x)$ for all $y\in K$.
\end{lemma}

\begin{proof}
For a proof see \cite[Theorem 8.73]{einsiedler2017functional}.
\end{proof}

\begin{lemma}\label{log}
If $x\neq y$ are positive real numbers, then 
\begin{align*}
\left|\log\dfrac{x}{y}\right|^{-1}<\dfrac{\max\lbrace x,y\rbrace}{|x-y|}.
\end{align*}
\end{lemma}

\begin{proof}
We give shortly the proof.
 Assume without loss of generality that $x>y$. 
 Then,
\begin{align*}
\log\dfrac{x}{y}
=-\log\dfrac{x+y-x}{x}=-\log\left(1-\dfrac{x-y}{x}\right)
=\mathlarger\sum\limits\limits_{n=1}^{\infty}\dfrac{1}{n}\left(\dfrac{x-y}{x}\right)^n
>\dfrac{x-y}{x}
\end{align*}
and the assertion of the lemma follows.
\end{proof}

\begin{lemma}\label{stand1}
For $T>0$ and $0<\sigma\neq1$
\begin{align*}
\mathlarger\sum\limits_{n\leq T}\dfrac{1}{n}
=\log T+\gamma+O\left(T^{-1}\right)
\,\,\,\text{ and }\,\,\,\,
\mathlarger\sum\limits_{n\leq T}\dfrac{1}{n^\sigma}
=\dfrac{T^{1-\sigma}}{1-\sigma}+\zeta(\sigma)+O\left(T^{-\sigma}\right),
\end{align*}
where $\gamma$ is the Euler-Mascheroni constant.
\end{lemma}

\begin{proof}
For a proof see \cite[Theorem 3.2]{zbMATH03523640}.
\end{proof}

\begin{lemma}\label{stand}
For $0<\alpha\leq1$, $1/2<\sigma\leq\sigma_0<1$ and $j=0,1$, we have 
\begin{align*}
\mathlarger\sum\limits_{1\leq m\neq n\leq T}\dfrac{1}{(m+\alpha)^\sigma(n+\alpha)^{\sigma}}\left|\log\dfrac{n+\alpha}{m+\alpha}\right|^{-j}
\ll_{\sigma_0} T^{2-2\sigma}\left(\log T\right)^j.
\end{align*} 
\end{lemma}

\begin{proof}
If $j=0$, then
\begin{align*}
\mathlarger\sum\limits_{1\leq m\neq n\leq T}\dfrac{1}{(m+\alpha)^\sigma(n+\alpha)^{\sigma}}
<\left(\mathlarger\sum\limits_{n\leq T}\dfrac{1}{n^\sigma}\right)^2
\ll_{\sigma_0}T^{2-2\sigma}.
\end{align*}
If $j=1$, then by Lemma \ref{log} we obtain that
\begin{align*}
\mathlarger\sum\limits_{1\leq m\neq n\leq T}\dfrac{1}{(m+\alpha)^\sigma(n+\alpha)^{\sigma}}\left|\log\dfrac{n+\alpha}{m+\alpha}\right|^{-1}
<4\mathlarger\sum\limits_{1\leq m< n\leq T}\dfrac{1}{m^\sigma n^{\sigma}}\dfrac{n}{n-m}.
\end{align*}
We split the sum on the right hand side of the latter inequality according to the cases $m<n/2$ and $n/2\leq m<n$.
We use Lemma \ref{stand1} to estimate the new sums. 
In the first case we have that
\begin{align*}
\mathlarger\sum\limits_{1< n\leq T}\,\mathlarger\sum\limits_{ m< \frac{n}{2}}\dfrac{1}{m^\sigma n^{\sigma}}\dfrac{n}{n-m}
\leq2\mathlarger\sum\limits_{1<n\leq T}\,\mathlarger\sum\limits_{ m< \frac{n}{2}}\dfrac{1}{m^\sigma n^{\sigma}}
<2\left(\mathlarger\sum\limits_{n\leq T}\dfrac{1}{n^\sigma}\right)^2
\ll_{\sigma_0}T^{2-2\sigma},
\end{align*}
while in the second case we set $m=n-r$ and we get that
\begin{align*}
\mathlarger\sum\limits_{1< n\leq T}\,\mathlarger\sum\limits_{ \frac{n}{2}\leq m<n}\dfrac{1}{m^\sigma n^{\sigma}}\dfrac{n}{n-m}
&<\mathlarger\sum\limits_{1<n\leq T}\,\mathlarger\sum\limits_{r\leq\frac{n}{2}}\dfrac{1}{(n-r)^\sigma n^{\sigma}}\dfrac{n}{r}\\
&\leq2\mathlarger\sum\limits_{n\leq T}\dfrac{1}{n^{2\sigma-1}}\mathlarger\sum\limits_{r\leq T}\dfrac{1}{r}\\
&\ll_{\sigma_0}T^{2-2\sigma}\log T.
\end{align*}
\end{proof}
We now present two lemmas regarding the order of the Hurwitz zeta-function in sufficiently {\it{narrow}} strips containing the vertical line $1+i\mathbb{R}$.

\begin{lemma}\label{sapprox}
If $0<\varepsilon<1$ then
\begin{align*}
\zeta_1(s;\alpha):=\zeta(s;\alpha)-\alpha^{-s}\ll_\varepsilon|t|^\varepsilon,\,\,\,\,\,|t|\geq 2,
\end{align*}
uniformly for $1-\varepsilon\leq\sigma\leq3$ and $0<\alpha\leq1$.
\end{lemma}

\begin{proof}
For a proof see \cite[Theorem 12.23]{zbMATH03523640}.
\end{proof}

The latter lemma does not give us a sufficiently good result regarding the order of the Hurwitz zeta-function, that is the exponent $\varepsilon$ on $|t|$ decreases linearly as $\varepsilon$ tends to 0 (and we approach the vertical line $1+i\mathbb{R}$ from the left). 
This will be seen to be insufficient to prove Lemma \ref{Perron1.}. 
The following lemma is a generalization of a well-known result among many results of the same spirit regarding the Riemann zeta-function.

\begin{lemma}\label{apprx}
The following bound
\begin{align}\label{boundL}
\zeta_1(s;\alpha)\ll |t|^{\eta(1-\sigma)^{3/2}}\log^{2/3} |t|,\,\,\,\,\,|t|\geq 3,
\end{align}
holds uniformly for $1/2\leq\sigma\leq1$ and $0<\alpha\leq1$, where $\eta=4.45$.
\end{lemma}

\begin{proof}
For a proof see \cite[Theorem 1]{zbMATH01884246}.
\end{proof}

The next lemma provides a representation of the Hurwitz zeta-function
 in suitable strips which include the vertical line $1+i\mathbb{R}$.
 Recall that $\lfloor x\rfloor$ denotes the largest integer which is less than or equal to the real number $x$.

\begin{lemma}\label{Perron1.}
For every $0<\mu<1$, there exists a positive number $\nu=\nu(\mu,\sigma_0)$, such that
\begin{align*}
\zeta(s;\alpha)
=\sum\limits_{0\leq n\leq t^\mu}\dfrac{1}{(n+\alpha)^s}+O_{\mu,\sigma_0}\left(t^{-\nu}\right),\,\,\,t\geq t_1>1,
\end{align*}
uniformly in $\mathbf{A}(\mu)<\sigma_0\leq\sigma\leq2$ and $0<\alpha\leq1$, where
\begin{align}\label{A}
\mathbf{A}(\mu):=1-\theta\mu^2,
\end{align}
$\theta=4/(27\eta^2)$ and $\eta=4.45$.
\end{lemma}

\begin{proof}
If we set  $c=1+b$, where $b=b(\mu)\in(0,1]$ will be determined later on, and $x=m+1/2$, $m\in\mathbb{N}$, then the absolute convergence of $\zeta_1(s;\alpha)$ in the half-plane $\sigma>1$ and Perron's formula  (see \cite[Lemma 12.1]{ivicriemann}) imply that
\begin{align}\label{Perron1}
\begin{split}
\dfrac{1}{2\pi i}\int\limits_{c-iT}^{c+iT}\zeta_1(s+z;\alpha)\dfrac{(x+\alpha)^z}{z}\mathrm{d}z
=&\,\sum\limits_{n=1}^{m}\dfrac{1}{(n+\alpha)^s}+\\
&+O_{\sigma_0}\hspace{-1.75pt}\left(\dfrac{1}{T}\sum\limits_{n=1}^{\infty}\left(\dfrac{x+\alpha}{n+\alpha}\right)^{c}\left|\log\dfrac{x+\alpha}{n+\alpha}\right|^{-1}\right),
\end{split}
\end{align}
uniformly in $\sigma\geq\sigma_0>0$ and $0<\alpha\leq1$.
We estimate the sum in the error term:
\begin{align}\label{sim1}
\begin{split}
\left\{\sum\limits_{n<\frac{x}{2}}+\sum\limits_{n>2x}\right\}\left(\dfrac{x+\alpha}{n+\alpha}\right)^c\left|\log\dfrac{x+\alpha}{n+\alpha}\right|^{-1}
&\ll x^c\hspace{-1.9pt}\left\{\sum\limits_{n<\frac{x}{2}}+\sum\limits_{n>2x}\right\}\dfrac{\max\lbrace x,n\rbrace+\alpha	}	{n^{c}|x-n|}\\
&\ll x^c\sum\limits_{n=1}^{\infty}\dfrac{1}{n^c}\\
&\ll\dfrac{x^c}{b},
\end{split}
\end{align}
while if we set $q=m-n$ for $x/2\leq n<x$ and $r=n-m$ for $x<n\leq2x$, we have
\begin{align}\label{sim2}
\begin{split}
\sum\limits_{\frac{x}{2}\leq n\leq2x}\left(\dfrac{x+\alpha}{n+\alpha}\right)^c\left|\log\dfrac{x+\alpha}{n+\alpha}\right|^{-1}
&\ll\sum\limits_{\frac{x}{2}\leq n\leq2x}\dfrac{x^c}{n^{c}}\dfrac{\max\lbrace x,n\rbrace}{|x-n|}\\
&\ll x\left[\sum\limits_{0\leq q\leq\frac{x-1}{2}}\dfrac{1}{q+\frac{1}{2}}+\sum\limits_{ r\leq\frac{2x+1}{2}}\dfrac{1}{r-\frac{1}{2}}\right]\\
&\ll x\log x
\end{split}
\end{align}
Hence, we deduce from (\ref{Perron1})-(\ref{sim2})  that
\begin{align}\label{Perron2}
\dfrac{1}{2\pi i}\int\limits_{c-iT}^{c+iT}\zeta_1(s+z;\alpha)\dfrac{(x+\alpha)^z}{z}\mathrm{d}z
=\sum\limits_{n=1}^{m}\dfrac{1}{(n+\alpha)^s}+O_{\sigma_0}\left(\dfrac{x^{c}}{bT}+\dfrac{x\log x}{T}\right),
\end{align}
uniformly in $\sigma\geq\sigma_0>0$ and $0<\alpha\leq1$.

Let  $1-\kappa\leq\sigma\leq2$ be arbitrary, where $\kappa=\kappa(\mu)\in[0,1/2]$ will be determined later on. 
Let also $T= 2t$ and consider the rectangle $\mathcal{R}$ with vertices $1-3\kappa-\sigma\pm iT$, $c\pm iT$. 
By the calculus of residues we get
\begin{align}\label{resid}
\begin{split}
\dfrac{1}{2\pi i}\int_{\mathcal{R}}\zeta_1(s+z;\alpha)\dfrac{(x+\alpha)^z}{z}\mathrm{d}z
=\zeta_1(s;\alpha)+\dfrac{(x+\alpha)^{1-s}}{1-s}=\zeta_1(s;\alpha)+O\left(\frac{x^{1-\sigma}}{t}\right).
\end{split}
\end{align} 
Observe that Lemma \ref{sapprox} implies that
\begin{align}\label{Lin2}
\left\{\int\limits_{1-3\kappa-\sigma-iT}^{c-iT}+\int\limits_{c+iT}^{1-3\kappa-\sigma+iT}\,\right\}\zeta_1(s+z;\alpha)\dfrac{(x+\alpha)^z}{z}\mathrm{d}z
\ll_\kappa \frac{x^{c}T^{3\kappa}}{T},
\end{align}
while Lemma \ref{apprx} yields
\begin{align}\label{Lin3}
\begin{split}
\int\limits_{1-3\kappa-\sigma+iT}^{1-3\kappa-\sigma-iT}\zeta_1(s+z;\alpha)\dfrac{(x+\alpha)^z}{z}\mathrm{d}z
&\ll x^{1-3\kappa-\sigma}\int\limits_{-T}^{T}\dfrac{\left|\zeta_1\left(1-3\kappa+i(t+u)\right)\right|}{\left|1-3\kappa+iu\right|}\mathrm{d}u\\
&\ll_\kappa x^{-2\kappa}T^{(3\kappa)^{3/2}\eta}\left(\log T\right)^2.
\end{split}
\end{align}
From relations (\ref{Perron2})-(\ref{Lin3}) we deduce
\begin{align*}
\zeta_1(s;\alpha)=&\,\sum\limits_{n=1}^{m}\dfrac{1}{(n+\alpha)^s}
+O_{\sigma_0}\left(\dfrac{x^{c}}{bT}+\dfrac{x\log x}{T}\right)+O\left(x^{1-\sigma}t^{-1}\right)+\\
&+O_\kappa\left(x^{c}T^{-1+3\kappa}+x^{-2\kappa}T^{(3\kappa)^{3/2}\eta}\left(\log T\right)^2\right).
\end{align*}
If we set $m=\lfloor t^\mu\rfloor$, then the last three terms in the latter relation are bounded above by
\begin{align*}
C(\sigma_0,\kappa,b)\hspace{-3pt}\left(\frac{t^{(1+b)\mu-1}}{b}
\hspace{-1pt}+\hspace{-1pt}t^{\mu(1-\sigma)-1}\hspace{-1pt}+\hspace{-1pt}
t^{(1+b)\mu+3\kappa-1}\hspace{-1pt}+\hspace{-1pt}t^{\kappa\left(-2\mu+3^{3/2}\kappa^{1/2}\eta\right)}\hspace{-1pt}\left(\log t\right)^2\right),
\end{align*}
where $C(\sigma_0,\kappa,b)>0$ is a constant.
It is clear now that for $
\kappa=4\mu^2/(27\eta^2)$ and $0<b\ll_\mu1
$ sufficiently small,
the lemma follows.
\end{proof}

 \begin{lemma}\label{approx'}
For every $d\geq3$ and $k\in\mathbb{N}_0$, there exists a positive number $\nu=\nu(d,k)$ such that
\begin{align*}
\zeta^{(k)}(s;\alpha)
=\sum\limits_{n=0}^{\left\lfloor t^{1/d}\right\rfloor}\dfrac{\left(-\log (n+\alpha)\right)^k}{(n+\alpha)^s}+O_{d,k}\left(t^{-\nu}\right),\,\,\,\,t\geq t_1>0,
\end{align*}
uniformly in $\mathbf{A}\left((d+1/(2d))^{-1}\right)\leq\sigma\leq1$ and $0<\alpha\leq1$.
\end{lemma}

\begin{proof}
Since $d\geq3$, we have from Lemma \ref{Perron1.} that  $\mathbf{A}(\mu)=1-\theta\mu^2$ for any $0<\mu\leq1/d$. In addition, there exists a positive number $\nu(d)$ such that
\begin{align*}
\zeta(s;\alpha)
=\sum\limits_{0\leq n\leq t^{1/d}}\dfrac{1}{(n+\alpha)^s}+O_{d}\left(t^{-\nu}\right),\,\,\,t\geq t_1>1,
\end{align*}
uniformly in $0<\alpha\leq1$ and 
$$\mathbf{A}\left(\dfrac{1}{d}\right)<\dfrac{1}{2}\left(\mathbf{A}\left(\left(d+\frac{1}{2d}\right)^{-1}\right)+\mathbf{A}\left(\dfrac{1}{d}\right)\right)\leq\sigma\leq2.$$
 Now the lemma follows by applying Cauchy's integral formula in the latter approximate functional equation for $\zeta(s;\alpha)$. 
\end{proof}

The next lemma originates from a work of G\"uting \cite{guting1967polynomials}.

\begin{lemma}\label{polyn}
Let $P(X)$ and $Q(X)$ be non-constant integer polynomials of degree $n$ and $m$, respectively. 
Denote by $\alpha$ a zero of $Q(X)$ of order $t$. 
Assuming that $P(\alpha)\neq0$, we have
\begin{align*}
|P(\alpha)|\geq(n+1)^{1-m/t}(m+1)^{-n/(2t)}H(P)^{1-m/t}H(Q)^{-n/t}(\max\lbrace1,|\alpha|\rbrace)^n
\end{align*}
\end{lemma}

\begin{proof}
For a proof see \cite[Theorem A.1]{bugeaud2004approximation}.
\end{proof}

\begin{lemma}\label{pol}
Let $P(X)$ be a non-zero integer polynomial of degree $n$ and $\alpha\in(0,1]$ be an algebraic number of degree $d(\alpha)$ and height $H(\alpha)$.
Assuming that $P(\alpha)\neq0$, we have
\begin{align*}
|P(\alpha)|\geq(n+1)^{1-d(\alpha)}(d(\alpha)+1)^{-n/2}H(P)^{1-d(\alpha)}H(\alpha)^{-n}
\end{align*}
\end{lemma}

\begin{proof}
Follows immediately from Lemma \ref{polyn}.
\end{proof}

\section{Two Auxiliary Theorems}\label{Auxil}

In the sequel we will use the abbreviation $\mathrm{e}(x)=\exp(2\pi ix)$ for $x\in\mathbb{R}$.  If $Q\in\mathbb{N}$,
we define the function 
\begin{align*}
\left(s,\underline{\theta},\alpha\right)\longmapsto\zeta_Q\left(s,\underline{\theta},\alpha\right):=\mathlarger\sum\limits_{n=0}^{Q-1}\dfrac{\mathrm{e}(\theta_n)}{(n+\alpha)^s},
\end{align*} 
for every $\left(s,\underline{\theta},\alpha\right)\in\mathbb{C}\times\mathbb{R}^Q\times(0,1]$.

We start with a modification of Good's Lemma 9 in \cite{good1980distribution} on the effective approximation of vectors of complex numbers by suitable twisted Dirichlet polynomials. 
An alternative option would be to follow Voronin's approach as in \cite[Chapter 8, Section 2, Lemma 1]{karatsuba1992riemann}. 
Interestingly enough, we could not deduce a result for $\sigma=1$ by the second way. 
And since also Good does not include the case of $\sigma=1$, we shall add it in our proof.

\begin{theorem}\label{MAIN1}
For every $\sigma\in(1/2,1]$ and $N\in\mathbb{N}$, there exist positive numbers $C_0$, $C_1$ and $C_2$, depending on $\sigma$ and $N$, such that the following is true:

 Let $A\in(0,1]$, $\varepsilon>0$ and $\textbf{a}=(a_0,\dots,a_{N})\in\mathbb{C}^{N+1}$. 
 Let also 
 $$R
 \geq C_0\,\varepsilon^{4/(1-2\sigma)}$$
  and $Q_0\geq C_1 R$ be integers satisfying the system of inequalities
\begin{align*}
C_2\left(|a_k|+A^{-1/2}\right)
\leq\mathbf{E}(R,Q_0,\sigma)\left(\dfrac{\log\frac{Q_0}{R+1}}{2N\log Q_0}\right)^{N}k!(N-k)!\left(\log Q_0\right)^k,
\end{align*}
$k=0,\dots,N$,
where
\begin{align}\label{E_def}
\mathbf{E}(R,Q,\sigma)
:=\left\{\begin{array}{lll}
\dfrac{R^{1-\sigma}}{2^{3+\sigma}(1-\sigma)}\left[\left(\dfrac{Q}{R+1}\right)^{(1-\sigma)/(4N^3)}-1\right]&,\sigma\neq1,\\
\\
\dfrac{\log\frac{Q}{R+1}}{2^5N^3}&,\sigma=1.
\end{array}\right.
\end{align}
 Then, for every $Q\geq Q_0$ and $\alpha\in[A,1]$, there exists $\underline{\theta}_0\in[0,1]^{Q}$ such that
\begin{align*}
\left|{\left.\dfrac{\partial^k}{\partial s^k}\zeta_Q\left(s,\underline{\theta}_0,\alpha\right)\right|}_{s=\sigma}-a_k\right|
<\varepsilon,\,\,\,\,k=0,\dots,N.
\end{align*}
\end{theorem}

\begin{proof}
Let  $R=R(\varepsilon,\sigma,N)$ be a positive integer which will be specified later on. We consider for every integer $Q>R$ the set of vectors
\begin{align*}
\mathcal{D}_{RQ}
:=\left\{\mathbf{z}=\left(z_R,\dots,z_{Q-1}\right):|z_n|\leq1,\,n=R,\dots,Q-1\right\}
\end{align*}
and define the functions 
\begin{align}\label{interm}
(\mathbf{z},\alpha)\longmapsto g_k(\mathbf{z},\alpha)
:=\sum\limits_{n=R}^{Q-1}z_n\dfrac{\left(-\log(n+\alpha)\right)^k}{(n+\alpha)^\sigma},
\end{align}
for every $(\mathbf{z},\alpha)\in\mathcal{D}_{RQ}\times(0,1]$ and $k=0,\dots ,N$.

First we will determine for a given vector of complex numbers $(A_0,\dots,A_{N})$ an integer $Q$ such that, for every $0<\alpha\leq1$, the system of equalities
\begin{align}\label{system}
g_k(\mathbf{z},\alpha)
=A_k,\,\,\,k=0,\dots, N,
\end{align}
has a solution $\mathbf{z}_\alpha\in\mathcal{D}_{RQ}$, that is, $(A_0,\dots,A_{N})$ belongs to the set
\begin{align*}
\mathcal{G}
:=\left\{\left(g_0(\mathbf{z},\alpha),\dots,g_{N}\left(\mathbf{z},\alpha\right)\right):\mathbf{z}\in\mathcal{D}_{RQ}\right\}.
\end{align*}
Observe that $\mathcal{G}$ is a closed convex subset of the complex Hilbert space $\mathbb{C}^{N+1}$ endowed with the inner product
\begin{align*}
\langle(x_0,\dots,x_{N}),(y_0,\dots,y_{N})\rangle
:=\sum\limits_{k=0}^{N}\Re(x_k\overline{y_k}).
\end{align*}
Thus, in view of Lemma \ref{Hil} it is sufficient to show that for sufficiently large $Q$ and for arbitrary $0<\alpha\leq1$ and non-zero $(\ell_0,\dots,\ell_{N})\in\mathbb{C}^{N+1}$, there is $\mathbf{z}\in\mathcal{D}_{RQ}$ such that
\begin{align}\label{system1}
\sum\limits_{k=0}^{N}\ell_kg_k(\mathbf{z},\alpha)
=\sum\limits_{k=0}^{N}\ell_kA_k.
\end{align}
One can see that
\begin{align}\label{disc}
\sum\limits_{k=0}^{N}\ell_kg_k(\mathcal{D}_{RQ},\alpha)
=\left\{z:|z|\leq V:=\sum\limits_{n=R}^{Q-1}\dfrac{1}{(n+\alpha)^\sigma}\left|\sum\limits_{k=0}^{N}\ell_k(-\log(n+\alpha))^k\right|\right\}.
\end{align}
Indeed, the inclusion of the set on the left-hand side in the set on the right-hand side is obvious, while if $w=|w|\mathrm{e}(\phi)$ belongs to the disc described in the right-hand side of (\ref{disc}), we can choose $\mathbf{z}\in\mathcal{D}_{RQ}$ with 
\begin{align*}
z_n
=\dfrac{|w|}{V} \mathrm{e}\left(\phi-\arg\left(\sum\limits_{m=0}^{N}\ell_m(-\log(n+\alpha))^m\right)\right)
\end{align*}
such that
\begin{align*}
\sum\limits_{k=0}^{N}\ell_kg_k(\mathbf{z},\alpha)=w.
\end{align*}
Therefore, from (\ref{system1}) and (\ref{disc}) it is sufficient to show that, for sufficiently large $Q$ and for arbitrary $0<\alpha\leq1$ and non-zero $(\ell_0,\dots,\ell_{N})\in\mathbb{C}^{N+1}$,
\begin{align}\label{cond}
\sum\limits_{k=0}^{N}|\ell_k||A_k|\leq\sum\limits_{n=R}^{Q-1}\dfrac{1}{(n+\alpha)^{\sigma}}\left|\sum\limits_{k=0}^{N}\ell_k(-\log(n+\alpha))^k\right|.
\end{align}

Now, consider the polynomial
\begin{align}\label{polynomial}
P(x):=\sum\limits_{k=0}^{N}(-1)^k\ell_kx^k,\,\,\,x\in\mathbb{R},
\end{align}
and the following partition of the interval $[\log (R+\alpha),\log Q]$
\begin{align}\label{partition}
x_k:=\log (R+\alpha)+\dfrac{k}{N}\log\dfrac{Q}{R+\alpha},\,\,\,k=0,\dots,N.
\end{align}
If we set in addition
\begin{align*}
G_k(x):=\mathop{\prod\limits_{m=0}^N}_{m\neq k}(x-x_m),\,\,\,k=0,\dots,N,
\end{align*}
then it follows that
\begin{align}\label{der}
\left|G_k^{(j)}(0)\right|\leq\mathop{\sum\limits_{m_1=0}^{N}}_{m_1\neq k}\mathop{\sum\limits_{m_2=0}^{N}}_{m_2\neq k,m_1}\dots\mathop{\sum\limits_{m_j=0}^{N}}_{m_j\neq k,m_1,\dots,m_{j-1}}\left|-x_{m_j}\right|\leq\dfrac{N!}{(N-j)!}(\log Q)^{N-j}
\end{align}
and
\begin{align}\label{poin}
\left|G_k(x_k)\right|=\mathop{\prod\limits_{m=0}^N}_{m\neq k}\left|\dfrac{k-m}{N}\log\dfrac{Q}{R+\alpha}\right|=\left(\dfrac{1}{N}\log\dfrac{Q}{R+\alpha}\right)^Nk!(N-k)!
\end{align}
for any $ j,k=0,\dots, N$. 
In view of Lagrange's interpolation theorem (see \cite[Chapter 1, Section 1,E.6]{borwein1995polynomials})
\begin{align*}
P(x)=\sum\limits_{k=0}^N\dfrac{P(x_k)}{G_k(x_k)}G_k(x),
\end{align*}
and relations \eqref{der} and \eqref{poin}, 
we obtain
\begin{align*}
j!|\ell_j|
=\left|P^{(j)}(0)\right|
&=\left|\sum\limits_{k=0}^N\dfrac{P(x_k)G^{(j)}_k(0)}{G_k(x_k)}\right|\\
&\leq\sum\limits_{k=0}^N\dfrac{|P(x_k)|N!(\log Q)^{N-j}}{k!(N-k)!(N-j)!}\left(\dfrac{N}{\log\frac{Q}{R+\alpha}}\right)^N
\end{align*}
for $j=0,\dots,N$.
Therefore,
\begin{align}\label{Good}
\dfrac{1}{N+1}\left(\dfrac{\log\frac{Q}{R+\alpha}}{2N\log Q}\right)^{N}\sum\limits_{j=0}^{N}j!|\ell_j|(N-j)!\left(\log Q\right)^j
\leq\sum\limits_{k=0}^{N}\left|P(x_k)\right|.
\end{align}

Let $y_k$, $k=1,\dots, N$, be such that $x_{k-1}\leq y_k\leq x_k$ and
\begin{align*}
|P(y_k)|=\max\limits_{x\in[x_{k-1},x_k]}|P(x)|= \max\limits_{x\in[-1,1]}\left|P\left(x\dfrac{x_k-x_{k-1}}{2}+\dfrac{x_k+x_{k-1}}{2}\right)\right|
\end{align*}
for $k=1,\dots, N$. 
Markov's inequality (see \cite[Chapter 5, Section 2, E.2]{borwein1995polynomials}) states that
$$\max\limits_{x\in[-1,1]}\left|\tilde{P}'(x)\right|\leq N^2\max\limits_{x\in[-1,1]}\left|\tilde{P}(x)\right|$$
 for any $\tilde{P}\in\mathbb{C}[X]$ of degree at most $N$.
Thus,
\begin{align}\label{Markov}
\begin{split}
\max\limits_{x\in[x_{k-1},x_k]}\left|P'(x)\right|
&=\max\limits_{x\in[-1,1]}\left|P'\left(x\dfrac{x_k-x_{k-1}}{2}+\dfrac{x_k+x_{k-1}}{2}\right)\right|\\
&=\max\limits_{x\in[-1,1]}\dfrac{2}{x_k-x_{k-1}}\left|\dfrac{\mathrm{d}}{\mathrm{d}x}P\left(x\dfrac{x_k-x_{k-1}}{2}+\dfrac{x_k+x_{k-1}}{2}\right)\right|\\
&\leq\dfrac{2N^2}{x_k-x_{k-1}}\max\limits_{x\in[-1,1]}\left|P\left(x\dfrac{x_k-x_{k-1}}{2}+\dfrac{x_k+x_{k-1}}{2}\right)\right|\\
&=\dfrac{2N^2}{x_k-x_{k-1}} |P(y_k)|
\end{split}
\end{align}
for $k=1,\dots,N$.
If we set now
\begin{align}\label{sets}
\mathcal{I}_k:=\left\{x\in[x_{k-1},x_k]:|x-y_k|\leq S:=\dfrac{\log\frac{Q}{R+\alpha}}{4N^3}\right\},\,\,\,k=1,\dots,N,
\end{align}
 then
relations \eqref{partition}, \eqref{Markov}, \eqref{sets} and the mean-value theorem imply that for every $x\in\mathcal{I}_k$ there is a $\xi_x$ between the points $x$ and $y_k$ such that
\begin{align*}
\left|P(x)\right|\geq |P(y_k)|-|P(y_k)-P(x)|=|P(y_k)|-\left|{P'(\xi_x)\left(y_k-x\right)}\right|\geq\dfrac{|P(y_k)|}{2}
\end{align*}
or
\begin{align}\label{inepo}
\max\limits_{x\in[x_{k-1},x_k]}|P(x)|=|P(y_k)|\leq 2\min\limits_{x\in[x_{k-1},x_k]k}|P(x)|
\end{align}
for $k=1,\dots, N$.
Since $$x_k-x_{k-1}=\dfrac{\log\frac{Q}{R+\alpha}}{N},\,\,\,k=1,\dots,N,$$
at least one of the intervals $[y_k-S,y_k]$ and $[y_k,y_k+S]$ is contained in $\mathcal{I}_k$.
 We denote those intervals by 
\begin{align}\label{J_k}
\mathcal{J}_k:=[c_k,c_k+S],\,\,\,k=1,\dots, N.
\end{align}
 Then, it follows from (\ref{polynomial}), (\ref{partition}), (\ref{inepo}) and (\ref{J_k}) that
\begin{align}\label{Good1}
\begin{split}
\sum\limits_{n=R}^{Q-1}\dfrac{1}{(n+\alpha)^\sigma}\left|\sum\limits_{k=0}^{N}\ell_k(-\log(n+\alpha))^k\right|&=\sum\limits_{n=R}^{Q-1}\dfrac{\left|P(\log(n+\alpha))\right|}{(n+\alpha)^\sigma}\\
&\geq\sum\limits_{k=1}^{N}\mathop{\sum\limits_{\log(n+\alpha)\in \mathcal{J}_k}}\dfrac{\left|P(\log(n+\alpha))\right|}{(n+\alpha)^{\sigma}}\\
&\geq\sum\limits_{k=1}^{N}\dfrac{|P(y_k)|}{2}\mathop{\sum\limits_{e^{c_k}\leq n+\alpha\leq e^{c_k+S}}}\dfrac{1}{(2n)^{\sigma}}.
\end{split}
\end{align}
Observe that
\begin{align*}
\mathop{\sum\limits_{e^{c_k}\leq n+\alpha\leq e^{c_k+S}}}\dfrac{1}{n^{\sigma}}\geq\left\{\begin{array}{lll}
\dfrac{e^{c_k(1-\sigma)}\left(e^{S(1-\sigma)}-1\right)}{1-\sigma}+O\left(e^{-c_k}\right)&,\sigma<1,\\
\\
\log\dfrac{e^{c_k+S}}{e^{c_k}}+O\left(e^{-c_k}\right)&,\sigma=1.
\end{array}\right.
\end{align*}
Since $c_k\geq\log R$ for $k=1,\dots,N$, the definition of $S$ yields that
\begin{align*}
\mathop{\sum\limits_{e^{c_k}\leq n+\alpha\leq e^{c_k+S}}}\dfrac{1}{n^{\sigma}}\geq\left\{\begin{array}{lll}
\dfrac{R^{1-\sigma}}{2(1-\sigma)}\left[\left(\dfrac{Q}{R+1}\right)^{(1-\sigma)/(4N^3)}-1\right]&,\sigma<1,\\
\\
\dfrac{\log\frac{Q}{R+1}}{8N^3}&,\sigma=1,
\end{array}\right.
\end{align*}
for sufficienty large  $R\gg1$and $Q\geq C_1 R$, where $C_1=C_1(\sigma,N)$. 
Recall that the right-hand side part of the latter inequality is equal to $2^{2+\sigma}\mathbf{E}(R,Q,\sigma)$. 
It follows now from relations \eqref{inepo} and \eqref{Good1} that
\begin{align}
\begin{split}
\sum\limits_{n=R}^{Q-1}\dfrac{1}{(n+\alpha)^\sigma}\left|\sum\limits_{k=0}^{N}\ell_k(-\log(n+\alpha))^k\right|
&\geq2\mathbf{E}(R,Q,\sigma)\sum\limits_{k=1}^{N}|P(y_k)|\\
&\geq\mathbf{E}(R,Q,\sigma)\sum\limits_{k=0}^{N}|P(x_k)|.
\end{split}
\end{align}
Thus, in view of relations (\ref{Good}) and (\ref{Good1}), if we choose $Q\geq C_1 R $ large enough so that the system of inequalities
\begin{align}\label{systine}
|A_k|\leq\mathbf{E}(R,Q,\sigma)\left(\dfrac{\log\frac{Q}{R+1}}{2N\log Q}\right)^{N}k!(N-k)!\left(\log Q\right)^k,\,\,\,k=0,\dots,N,
\end{align}
is satisfied, then relation (\ref{cond}) holds for arbitrary $\alpha\in(0,1]$ and any non-zero vector $(\ell_0,\dots,\ell_N)\in\mathbb{C}^{N+1}$.
Hence, for every $\alpha\in(0,1]$ the system (\ref{system}) has a solution $\mathbf{z}_\alpha\in\mathcal{D}_{RQ}$ as long as $Q\geq C_ 1R$ satisfies (\ref{systine}).

If $U_0\gg_{\sigma,N}1$ is large enough so that the functions
\begin{align}\label{U}
x\longmapsto\dfrac{\left(\log x\right)^k}{x^\sigma},\,\,\,k=0,\dots,N,
\end{align}
are decreasing in $[U_0,+\infty]$, then for every $R>U_0$, $\alpha\in[A,1]$ and $k=0,\dots,N$, we have that
\begin{align*}
\left|\sum\limits_{n=0}^{R-1}\dfrac{(-1)^n(-\log(n+\alpha))^k}{(n+\alpha)^{\sigma}}\right|
&\leq\dfrac{(-\log\alpha)^{k}}{\alpha}+\left|\sum\limits_{n=1}^{U}\dfrac{(-1)^n\left(\log(n+\alpha)\right)^k}{(n+\alpha)^{\sigma}}\right|\\
&\leq\dfrac{(-\log\alpha)^{k}}{\alpha^{\sigma}}+\max\limits_{y\in[0,1]}\left|\sum\limits_{n=1}^{U}\dfrac{(-1)^n\left(\log(n+y)\right)^k}{(n+y)^{\sigma}}\right|\\
&\leq{C_2}A^{-1/2},
\end{align*}
where $C_2=C_2(\sigma,N)\geq1$.
Therefore, if we set 
\begin{align*}
A_k=A_k(\alpha):=a_k-\sum\limits_{n=0}^{R-1}\dfrac{(-1)^n(-\log(n+\alpha))^k}{(n+\alpha)^{\sigma}},\,\,\,k=0,\dots,N,
\end{align*}
it follows from (\ref{systine}) that for every $\alpha\in[A,1]$ the system of equalities (\ref{system}) has a solution $\mathbf{z}_\alpha\in\mathcal{D}_{RQ}$ as long as $Q\geq C_1 R$ satisfies the system of inequalities
\begin{align*}
C_2\left(|a_k|+A^{-1/2}\right)
\leq\mathbf{E}(R,Q,\sigma)\left(\dfrac{\log\frac{Q}{R+1}}{2N\log Q}\right)^{N}k!(N-k)!(\log Q)^k,
\end{align*}
$k=0,\dots,N$.
Since the right-hand side of these inequalities tends to infinity as $Q\to\infty$, the system is solvable for all sufficiently large $Q$.

Let $Q_0\geq C_1 R$ be the smallest integer satisfying the aforementioned system, $Q\geq Q_0$ and $\alpha\in[A,1]$. 
Let also $\mathbf{z}_\alpha:=\left(z_n\right)_{R\leq n\leq Q_0-1}$ be an element of $\mathcal{D}_{RQ_0}$ such that
\begin{align}\label{system1.}
g_k(\mathbf{z}_\alpha,\alpha)=A_k(\alpha),\,\,\,k=0,\dots, N.
\end{align}
 From Lemma \ref{Hilbert} there are real numbers $\theta_n$, $n=R,\dots, Q-1$, such that 
\begin{align}\label{ine3}
\begin{split}
&\left|\left|\left(\sum\limits_{n=R}^{Q_0-1}z_{n}\dfrac{\left(-\log(n+\alpha)\right)^{k}}{(n+\alpha)^{\sigma}}-\sum\limits_{n=R}^{Q-1}\dfrac{\left(-\log(n+\alpha)\right)^{k}\mathrm{e}\left(\theta_{n}\right)}{(n+\alpha)^{\sigma}}\right)_{0\leq k\leq N}\right|\right|_{\mathbb{C}^{N+1}}^2\\
&\leq4\sum\limits_{n=R}^{Q-1}\left|\left|\left(\dfrac{\left(-\log(n+\alpha)\right)^{k}}{(n+\alpha)^\sigma}\right)_{0\leq k\leq N}\right|\right|_{\mathbb{C}^{N+1}}^2\\
&\leq4\sum\limits_{k=0}^{N}\sum\limits_{n=R}^{Q-1}\dfrac{\left(\log(n+1)\right)^{2k}}{n^{2\sigma}}\\
&\ll_{\sigma,N}R^{1-2\sigma}.
\end{split}
\end{align}
Let
$$R\gg_{\sigma,N}\left(U_0+\dfrac{1}{\varepsilon}\right)^{\frac{4}{2\sigma-1}}
\gg_{\sigma,N}\left(\dfrac{1}{\varepsilon}\right)^{\frac{4}{2\sigma-1}}$$
be
 sufficiently large
and set $\underline{\theta}_0=(\theta_{0n})_{0\leq n\leq Q-1}$ to be
\begin{align*}
\theta_{0n}:=\left\{\begin{array}{ll}
n/2,&0\leq n\leq R-1,\\
\theta_{n},&R\leq n\leq Q-1.
\end{array}\right.
\end{align*}
Then \eqref{interm}, \eqref{system1.} and \eqref{ine3} yield
\begin{align*}
\left|{\left.\dfrac{\partial^k}{\partial s^k}\zeta_Q\left(s,\underline{\theta}_0,\alpha\right)\right|}_{s=\sigma}-a_k\right|
&=\left|A_k(\alpha)-\sum\limits_{n=R}^{Q-1}\dfrac{\left(-\log(n+\alpha)\right)^{k}\mathrm{e}\left(\theta_{n}\right)}{(n+\alpha)^{\sigma}}\right|\\
&<\left|g_k(\mathbf{z}_\alpha,\alpha)\hspace*{-1pt}-\hspace*{-1pt}\sum\limits_{n=R}^{Q_0-1}z_{n}\dfrac{\left(-\log(n+\alpha)\right)^{k}}{(n+\alpha)^{\sigma}}\right|+\varepsilon\\
&=\varepsilon
\end{align*}
for $k=0,\dots,N$.
\end{proof}

Before proving the next theorem, we need to introduce some notations.
Let ${{\lambda:\mathbb{R}\to\mathbb{R}_+}}$
be an infinitely differentiable function with $\mathrm{supp}(\lambda)\hspace*{-1.1pt}\subseteq\left[-1,1\right]$ and $\int_{-\infty}^{+\infty}\lambda(x)\mathrm{d}x=1$. 
We also assume that $\lambda$ is bounded above by 1.
If $Q\geq2$ is an integer, we set $\delta:=Q^{-2}$ and define the function
\begin{align*}
{{\underline{\theta}\longmapsto \Lambda_Q(\underline{\theta}):=\prod\limits_{n=0}^{Q-1}\lambda\left(\dfrac{\theta_n}{\delta}\right)}},
\end{align*}
for any $\underline{\theta}=(\theta_0,\dots,\theta_{Q-1})\in\left[-1,1\right]^Q$. 
Then, $\mathrm{supp}\left(\Lambda_Q\right)\subseteq\left[-1/2,1/2\right]^Q$ and we can extend $\Lambda_Q$ onto all $\mathbb{R}^Q$ by periodicity with period 1 in each of the variables $\theta_n$, $n=0,\dots,Q-1$.
The function
 $${{\theta\longmapsto\lambda\left(\dfrac{\theta}{\delta}\right)}}$$ 
 extended to $\mathbb{R}$ by periodicity with period 1, has a Fourier expansion
\begin{align*}
{{\lambda\left(\dfrac{\theta}{\delta}\right):=\sum\limits_{n=-\infty}^{+\infty}{c}_n\mathrm{e}(n\theta)}},
\end{align*}
where 
\begin{align}\label{ineq3}
{{c_0={\delta}\,\,\,\,\,\text{and}\,\,\,\,\,
{c}_n=\int\limits_0^1\lambda\left(\dfrac{\theta}{\delta}\right)\mathrm{e}(-n\theta)\mathrm{d}\theta\ll\dfrac{1}{n^2\delta^{{2}}},\,\,\,n\in\mathbb{Z}\setminus\lbrace0\rbrace}}.
\end{align}
The last relation follows from integrating twice by parts and the implicit constant depends only on our choice of $\lambda$.
Thus, the Fourier expansion of $\Lambda_Q$ is given by
\begin{align*}
\Lambda_Q(\underline{\theta}):=\sum\limits_{\underline{m}}d_{\underline{m}}\mathrm{e}(\langle\underline{m},\underline{\theta}\rangle),
\end{align*}
where $\underline{m}=(m_0,\dots,m_{Q-1})\in\mathbb{Z}^Q$ and  $$d_{\underline{m}}:=\prod\limits_{n=0}^{Q-1}{c}_{m_n}.$$
We define for every $\underline{m}\in\mathbb{Z}^Q\setminus\lbrace\underline{0}\rbrace$ and $x\in\mathbb{R}$ the polynomials
\begin{align*}
Q_{\underline{m}}^+(x):=\mathop{\prod\limits_{n=0}^{Q-1}}_{m_n>0}(n+x)^{m_n}
\,\,\,\,\,\text{and}\,\,\,\,Q_{\underline{m}}^-(x):=\mathop{\prod\limits_{n=0}^{Q-1}}_{m_n<0}(n+x)^{-m_n}.
\end{align*}
Let $\hat{M}:=\mathbb{Z}^Q\cap[-M\,,\,M\,]^Q$ and
\begin{align}\label{poli}
\mathcal{P}(Q,M):=\left\{P_{\underline{m}}=Q_{\underline{m}}^+-Q_{\underline{m}}^-:\underline{m}\in\hat{M}\setminus\lbrace\underline{0}\rbrace\right\}.
\end{align}
Observe that $\mathcal{P}(Q,M)$ is a set of non-zero integer polynomials of degree at most $MQ$ and height bounded by a constant $\mathbf{H}(Q,M)$.
We also define the set 
\begin{align}\label{al}
\mathcal{A}(Q,M)=\mathcal{A}_1\cup\mathcal{A}_2,
\end{align} where
\begin{align*}
\mathcal{A}_1&:=\left\{\alpha\in\mathbb{A}:d(\alpha)>MQ+1\right\},
\end{align*}
\begin{align*}
\mathcal{A}_{2\underline{m}}
:=\left\{\alpha\in\mathbb{A}\setminus\mathcal{A}_1:\forall x,y\in\mathbb{N}\cap\left[0,\exp\left(2Q^2\right)\right],\,\,\dfrac{Q_{\underline{m}}^+(\alpha)}{Q_{\underline{m}}^-(\alpha)}\neq\dfrac{x+\alpha}{y+\alpha}\right\},
\end{align*}
for $\underline{m}\in\hat{M}\setminus\lbrace0\rbrace$, and
\begin{align*}
\mathcal{A}_2
:=\bigcap\limits_{\underline{m}\in\hat{M}\setminus\lbrace\underline{0}\rbrace}\mathcal{A}_{2\underline{m}}.
\end{align*}
Finally, we  consider the curve
\begin{align*}
\mathbb{R}\times(0,1]\ni(\tau,\alpha)\longmapsto\gamma_{Q}(\tau,\alpha):=\left(\dfrac{\log\left(n+\alpha\right)}{2\pi}\tau\right)_{0\leq n< Q}.
\end{align*}

\begin{theorem}\label{MAIN3}
For any $k\in\mathbb{N}_0$ and $d\geq3$, there exist positive numbers $C_3=C_3(k)$, $C_4$ and $C_5(d,k)$, such that the following is true:

Let $\varepsilon>0$, $Q\geq C_3/\varepsilon^{8}$, $M\geq C_4 \exp\left(2Q^2\right)$, $\alpha\in\mathcal{A}(Q,M)$ and $d\geq d(\alpha)+1$. 
Then there exists positive number $\nu=\nu(d,k)$, such that if 
\begin{align*}
{{T\geq C_5\max\left\{\left(\mathbf{K}\exp\left((M+2)\exp\left(Q^2\right)\right)\right)^{\frac{4d}{4(d-d(\alpha))-3}},\varepsilon^{-2\nu}\right\}}},
\end{align*}
where
\begin{align}\label{B}
\mathbf{K}=\mathbf{K}(Q,M,\alpha)
:=\left[\mathbf{H}(Q,M)\left(MQ+2\right)\right]^{d(\alpha)-1}\left[H(\alpha)(d(\alpha)+1)^{1/2}\right]^{MQ+1},
\end{align}
we have
\begin{align*}
{\left|\dfrac{1}{\delta^QT}\int\limits_{T}^{2T}\Lambda(\gamma_{Q}(\tau,\alpha)-\underline{\theta}_1)\mathrm{d}\tau-1\right|<Q^{-2}}
\end{align*} and
\begin{align}\label{bas3}
\begin{split}
\int\limits_{T}^{2T}&\Lambda_Q\left(\gamma_{Q}(\tau,\alpha)-\underline{\theta}_1\right)\left|\zeta^{(k)}\left(\sigma+i\tau;\alpha\right)-\left.\dfrac{\partial^k}{\partial s^k}\zeta_Q\left(s+i\tau,\underline{0},\alpha\right)\right|_{s=\sigma}\right|^2\mathrm{d}\tau\\
&< {\varepsilon^2}\int\limits_{T}^{2T}\Lambda_Q\left(\gamma_{Q}(\tau,\alpha)-\underline{\theta}_1\right)\mathrm{d}\tau
\end{split}
\end{align}
for any $\underline{\theta}_1\in\mathbb{R}^Q$ and $\mathbf{A}\left((d+1/(2d))^{-1}\right)\leq\sigma\leq1$.
\end{theorem}

\begin{proof}
First, we will show that 
\begin{align*}
\left|\dfrac{1}{\delta^QT}\int\limits_{T}^{2T}\Lambda_Q(\gamma_{Q}(\tau,\alpha)-\underline{\theta}_1)\mathrm{d}\tau-1\right|< Q^{-2}
\end{align*}
for suitable $Q$, $\alpha$, $T$ and any $\underline{\theta}_1\in\mathbb{R}^Q$.
The Fourier expansion of the function
\begin{align*}
\underline{\theta}\longmapsto \Lambda_Q(\underline{\theta}-\underline{\theta}_1)
\end{align*} 
is given by
\begin{align*}
\Lambda_Q(\underline{\theta}-\underline{\theta}_1):=\sum\limits_{\underline{m}}h_{\underline{m}}\mathrm{e}(\langle\underline{m},\underline{\theta}\rangle),
\end{align*}
where
$
h_{\underline{0}}:=\delta^Q
$
and
$h_{\underline{m}}:=d_{\underline{m}}\mathrm{e}(\langle\underline{m},-\underline{\theta}_1\rangle)
$,
$\underline{m}\in\mathbb{Z}^Q.$
For $M\in\mathbb{N}$,
\begin{align}\label{tri ineq}
\left|\sum\limits_{\underline{m}\notin\hat{M}}h_{\underline{m}}\mathrm{e}(\langle\underline{m},\underline{\theta}\rangle)\right|\leq\sum\limits_{\underline{m}\notin\hat{M}}\left|h_{\underline{m}}\right|
\leq Q\left(\sum\limits_{|n|>M}|{c}_n|\right)\left(\sum\limits_{n=-\infty}^{+\infty}|{c}_n|\right)^{Q-1}.
\end{align}
From (\ref{ineq3}) we know that
\begin{align}\label{coineq}
\sum\limits_{|n|>M}|{c}_n|\ll\dfrac{1}{\delta^{{2}}M}\,\,\,\,\,\text{ and }\,\,\,\,\,\sum\limits_{n=-\infty}^{+\infty}|{c}_n|\leq\left(\dfrac{A}{\delta}\right)^{{2}},
\end{align}
where $A>1$ is an absolute costant.
Therefore, from (\ref{tri ineq}) we conclude that 
\begin{align}\label{tri ineq 2}
\Lambda_Q(\underline{\theta}-\underline{\theta}_1)
=\sum\limits_{\underline{m}\in\hat{M}}h_{\underline{m}}\mathrm{e}(\langle\underline{m},\underline{\theta}\rangle)+O\left(\dfrac{Q}{M}\left(\dfrac{A}{\delta}\right)^{{{2}}Q}\right).
\end{align}
Observe that by $\delta= Q^{-2}$ we have
\begin{align*}
Q\left(\dfrac{A}{\delta}\right)^{{{2}}Q}\leq{{Q\delta^Q\left(\dfrac{A}{\delta}\right)^{{{3}}Q}\leq \delta^QQ\left(AQ\right)^{{{6}}Q}\ll\delta^Q\exp\left(Q^2\right).}}
\end{align*}
Hence, relation (\ref{tri ineq 2}) can be written as
\begin{align}\label{LL}
\Lambda_Q(\underline{\theta}-\underline{\theta}_1)=\sum\limits_{\underline{m}\in\hat{M}}h_{\underline{m}}\mathrm{e}(\langle\underline{m},\underline{\theta}\rangle)+O\left(\dfrac{\delta^Q\exp\left(Q^2\right)}{M}\right).
\end{align}
In the sequel we use the notations
\begin{align*}
\ell(\tau):=\Lambda_Q\left(\gamma_Q(\tau,\alpha)-\underline{\theta}_1\right)\,\,\,\,\,\text{and}\,\,\,\,\,\tilde{n}:=n+\alpha
\end{align*}
in order to avoid extensive expressions.
In view of (\ref{LL}), we have
{{\begin{align*}
\int\limits_{T}^{2T}\ell(\tau)\mathrm{d}\tau
=&\,h_{\underline{0}}T+\sum\limits_{\underline{m}\in\hat{M}\setminus\lbrace\underline{0}\rbrace}h_{\underline{m}}\int\limits_{T}^{2T}\mathrm{e}(\langle\underline{m},\gamma_{Q}(\tau,\alpha)\rangle)+O\left(\dfrac{T\delta^Q\exp\left(Q^2\right)}{M}\right),
\end{align*}}}
or
{{\begin{align}
\begin{split}\label{L1}
\dfrac{1}{\delta^QT}\int\limits_{T}^{2T}\ell(\tau)\mathrm{d}\tau
\hspace*{-0.5pt}=\hspace*{-0.5pt}1+\dfrac{1}{\delta^QT}\sum\limits_{\underline{m}\in\hat{M}\setminus\lbrace\underline{0}\rbrace}h_{\underline{m}}
\int\limits_{T}^{2T}\left(\dfrac{Q^+_{\underline{m}}(\alpha)}{Q^-_{\underline{m}}(\alpha)}\right)^{i\tau}\mathrm{d}\tau+O\left(\dfrac{\exp\left(Q^2\right)}{M}\right).
\end{split}
\end{align}}}
It follows from the definition of $h_{\underline{m}}$ and (\ref{coineq}) that
\begin{align}\label{coef}
\sum\limits_{\underline{m}}|h_{\underline{m}}|
\leq\left(\dfrac{A}{\delta}\right)^{2Q}\leq\delta^Q\left(AQ\right)^{6Q}
\ll \delta^Q\exp\left(Q^2\right).
\end{align}
It also follows from  (\ref{poli}) and (\ref{al}) that if $\underline{m}\in\hat{M}\setminus\lbrace\underline{0}\rbrace$ and $\alpha\in\mathcal{A}(Q,M)$, then $P_{\underline{m}}(\alpha)=Q^+_{\underline{m}}(\alpha)-Q^-_{\underline{m}}(\alpha)\neq0$. 
Thus,  it follows from Lemma \ref{log} that
\begin{align}\label{dif}
\int\limits_{T}^{2T}\left(\dfrac{Q^+_{\underline{m}}(\alpha)}{Q^-_{\underline{m}}(\alpha)}\right)^{i\tau}\mathrm{d}\tau\ll\left|\log\dfrac{Q^+_{\underline{m}}(\alpha)}{Q^-_{\underline{m}}(\alpha)}\right|^{-1}\ll\dfrac{\max\lbrace Q^+_{\underline{m}}(\alpha),Q^-_{\underline{m}}(\alpha)\rbrace}{\left|Q^+_{\underline{m}}(\alpha)-Q^-_{\underline{m}}(\alpha)\right|}.
\end{align}
Now Lemma \ref{pol} yields that, for every $\underline{m}\in\hat{M}\setminus\lbrace\underline{0}\rbrace$ and $\alpha\in\mathcal{A}(Q,M)$,
\begin{align}\label{low}
\begin{split}
\left|Q^+_{\underline{m}}(\alpha)-Q^-_{\underline{m}}(\alpha)\right|
&\geq\left[\mathbf{H}(Q,M)\left(MQ+1\right)\right]^{1-d(\alpha)}\left[H(\alpha)(d(\alpha)+1)^{1/2}\right]^{-MQ}\\
&>\mathbf{K}^{-1}.
\end{split}
\end{align}
Along with the estimate
\begin{align}\label{exp}
\max\lbrace Q^+_{\underline{m}}(\alpha),Q^-_{\underline{m}}(\alpha)\rbrace\ll{\prod\limits_{n=1}^{Q}}n^M\ll\exp\left(MQ^2\right),
\end{align}
we conclude from (\ref{L1})-(\ref{exp}) that
\begin{align*}
\dfrac{1}{\delta^QT}\int\limits_{T}^{2T}\ell(\tau)\mathrm{d}\tau-1
\ll\dfrac{\exp\left(Q^2\right)}{M}+\dfrac{\mathbf{K}\exp\left((M+1)Q^2\right)}{T}.
\end{align*}
For {{$Q\gg1$, $M\gg\exp\left(2Q^2\right)$, $\alpha\in\mathcal{A}(Q,M)$ and $T\gg\mathbf{K}\exp\left((M+2)Q^2\right)$}}, with suitable constants in $\gg$, we obtain 
\begin{align}\label{bas2}
\left|\dfrac{1}{\delta^QT}\int\limits_{T}^{2T}\Lambda_Q(\gamma_{Q}(\tau,\alpha)-\underline{\theta}_1)\mathrm{d}\tau-1\right|< Q^{-2}
\end{align}

We proceed now with the proof of relation (\ref{bas3}). 
Let $I$ denote the left-hand side of  (\ref{bas3}). 
Let also $\alpha\in\mathcal{A}(Q,M)$ and $d\geq d(\alpha)+1$. 
It follows from Lemma \ref{approx'} that there exists a positive number $\nu=\nu(d,k)$, such that
\begin{align*}
\zeta^{(k)}(s;\alpha)=\sum\limits_{n=0}^{\lfloor t^{1/d}\rfloor}\dfrac{\left(-\log (n+\alpha)\right)^k}{(n+\alpha)^s}+O_{d,k}\left(t^{-\nu}\right),\,\,\,\,t\geq t_1>0,
\end{align*}
uniformly in $\mathbf{A}\left(\left(d+1/(2d)\right)^{-1}\right)\leq\sigma\leq1$ and $0<\alpha\leq1$.
By substituting this approximate functional equation in $I$ for sufficiently large $
T\gg_{d} Q$, and by setting $p(\tau):=\left\lfloor \tau^{1/d}\right\rfloor$, it follows
 that
$
I\ll{I}_1+{I}_2,
$
where
\begin{align*}
{I}_1=\int\limits_{T}^{2T}\ell(\tau)
\left|\sum\limits_{n=Q}^{p(\tau)}\dfrac{(-\log\tilde{n})^k}{\tilde{n}^{\sigma+i\tau}}\right|^2\mathrm{d}\tau
\,\,\,\,\text{ and }\,\,\,\,
{I}_2=\int\limits_{T}^{2T}\ell(\tau)\,O_{d,k}\left(\tau^{-\nu}\right)\mathrm{d}\tau.
\end{align*}
Thus, it suffices to prove the theorem for ${I}_1$ and ${I}_2$.
We start by estimating ${I}_1$:
\begin{align}\label{series est.}
\begin{split}
{I}_1
\leq&\left(\delta^Q+O\left(\dfrac{\delta^Q\exp\left(Q^2\right)}{M}\right)\right)\int\limits_{T}^{2T}\left|\sum\limits_{n=Q}^{p(\tau)}\dfrac{(-\log\tilde{n})^k}{\tilde{n}^{\sigma+i\tau}}\right|^2\mathrm{d}\tau+\\
&+\left|\sum\limits_{\underline{m}\in\hat{M}\setminus\lbrace\underline{0}\rbrace}h_{\underline{m}}\int\limits_{T}^{2T}\mathrm{e}\left(\langle\underline{m},\gamma_Q(\tau,\alpha)\rangle\right)
\left|\sum\limits_{n=Q}^{p(\tau)}\dfrac{(-\log\tilde{n})^k}{\tilde{n}^{\sigma+i\tau}}\right|^2\mathrm{d}\tau\right|\\
\leq&\,\left[\delta^Q+O\left(\dfrac{\delta^Q\exp\left(Q^2\right)}{M}\right)\right]S_1+\sum\limits_{\underline{m}\in\hat{M}\setminus\lbrace\underline{0}\rbrace}\left|h_{\underline{m}}\right|\left|S_{2\underline{m}}\right|.
\end{split}
\end{align}
We estimate each of the terms on the right-hand side of (\ref{series est.}) seperately. By interchanging integration and summation we obtain 
\begin{align*}
S_1=\sum\limits_{n=Q}^{p(2T)}\dfrac{\left(\log\tilde{n}\right)^{2k}}{\tilde{n}^{2\sigma}}\int\limits_{T_1}^{2T}\mathrm{d}\tau
+ \sum\limits_{Q\leq n_1\neq n_2\leq p(2T)}\dfrac{(\log\tilde{n}_1\log\tilde{n}_2)^k}{\tilde{n}_1^{\sigma}\tilde{n}_2^{\sigma}}\int\limits_{T_2}^{2T}
\left(\dfrac{\tilde{n}_2}{\tilde{n}_1}\right)^{i\tau}\mathrm{d}\tau,
\end{align*}
where $T_1=\max\left\{T,\tilde{n}^d\right\}$ and $T_2=\max\left\{T,\tilde{n}_1^d,\tilde{n}_2^d\right\}$.
Since $\alpha\in\left(0,1\right]$, $d\geq3$ and $\sigma\geq \mathbf{A}(1/(d+1/(2d)))>3/4$, we get
\begin{align}\label{bas1}
\sum\limits_{n=Q}^{p(2T)}\dfrac{\left(\log\tilde{n}\right)^{2k}}{\tilde{n}^{2\sigma}}
\ll_k\sum\limits_{n=Q}^{\infty}\dfrac{\left(\log n\right)^{2k}}{n^{3/2}}
\ll_k Q^{-1/2}\left(\log Q\right)^{2k}\ll_kQ^{-1/4}
\end{align}
and
\begin{align}\label{bas.2}
\sum\limits_{Q\leq n_1\neq n_2\leq p(2T)}\dfrac{(\log\tilde{n}_1\log\tilde{n}_2)^k}{\tilde{n}_1^{\sigma}\tilde{n}_2^{\sigma}}
\ll_k \sum\limits_{Q\leq n_1\neq n_2\leq p(2T)}\dfrac{\left(\log p(2T)\right)^{2k}}{\left(\tilde{n}_1\tilde{n}_2\right)^{3/4}}.
\end{align}
Therefore,
\begin{align}
S_1
&\ll_k \,Q^{-1/4}T+\sum\limits_{Q\leq n_1\neq n_2\leq p(2T)}\dfrac{\left(\log p(2T)\right)^{2k}}{\left(\tilde{n}_1\tilde{n}_2\right)^{3/4}}\left|\int\limits_{T_2}^{2T}
\left(\dfrac{\tilde{n}_2}{\tilde{n}_1}\right)^{i\tau}\mathrm{d}\tau\right|\tag*{}\\
&\ll_k\,Q^{-1/4}T+\left(\log p(2T)\right)^{2k}\sum\limits_{Q\leq n_1\neq n_2\leq p(2T)}\dfrac{1}{\left(\tilde{n}_1\tilde{n}_2\right)^{3/4}}\left|\log\dfrac{\tilde{n}_2}{\tilde{n}_1}\right|^{-1}\tag*{}\\
&\ll_k\,Q^{-1/4}T+p(2T)^{1/2}\left(\log p(2T)\right)^{1+2k}\tag*{}\\
\label{S_1}&\ll_k\,Q^{-1/4}T+p(2T)^{3/4}.
\end{align}
For the second sum we have by interchanging integration and summation 
\begin{align}
S_{2\underline{m}}
=&\sum\limits_{n=Q}^{p(2T)}\dfrac{\left(\log \tilde{n}\right)^{2k}}{\tilde{n}^{2\sigma}}\int\limits_{T_1}^{2T}\left(\dfrac{Q^+_{\underline{m}}(\alpha)}{Q^-_{\underline{m}}(\alpha)}\right)^{i\tau}\mathrm{d}\tau\,+\tag*{}\\
\label{S_2.}&+ \sum\limits_{Q\leq n_1\neq n_2\leq p(2T)}\dfrac{(\log\tilde{n}_1\log\tilde{n}_2)^k}{\tilde{n}_1^{\sigma}\tilde{n}_2^{\sigma}}\int\limits_{T_2}^{2T}
\left(\dfrac{Q^+_{\underline{m}}(\alpha)\tilde{n}_2}{Q^-_{\underline{m}}(\alpha)\tilde{n}_1}\right)^{i\tau}\mathrm{d}\tau.
\end{align}
Here we consider two subcases, depending on whether $\alpha\in\mathcal{A}_1$ or $\alpha\in\mathcal{A}_2$.
It follows from the definitions in (\ref{poli}) and (\ref{al}) that, if  $\underline{m}\in\hat{M}\setminus\lbrace\underline{0}\rbrace$ and $\alpha\in\mathcal{A}_1$, then
 \begin{align}\label{Q con}Q^+_{\underline{m}}(\alpha)-Q^-_{\underline{m}}(\alpha)\neq0\,\,\,\text{ and }\,\,\,Q^+_{\underline{m}}(\alpha)\tilde{n}_2-Q^-_{\underline{m}}(\alpha)\tilde{n}_1\neq0.
 \end{align}  
Thus, applying Lemma \ref{pol}, it follows similar as in (\ref{dif})-\eqref{exp} that
\begin{align}\label{1}
\int\limits_{T_1}^{2T}\left(\dfrac{Q^+_{\underline{m}}(\alpha)}{Q^-_{\underline{m}}(\alpha)}\right)^{i\tau}\mathrm{d}\tau\ll\mathbf{K}\exp\left(MQ^2\right)
\end{align}
and
\begin{align}\label{2}
\begin{split}
\int\limits_{T_1}^{2T}\left(\dfrac{Q^+_{\underline{m}}(\alpha)\tilde{n}_2}{Q^-_{\underline{m}}(\alpha)\tilde{n}_1}\right)^{i\tau}\mathrm{d}\tau
&\ll\dfrac{\max\lbrace Q^+_{\underline{m}}(\alpha)\tilde{n}_2,Q^-_{\underline{m}}(\alpha)\tilde{n}_1\rbrace}{\left|Q^+_{\underline{m}}(\alpha)\tilde{n}_2-Q^-_{\underline{m}}(\alpha)\tilde{n}_1\right|}\\
&\ll\mathbf{K}\,p(2T)^{d(\alpha)}\exp\left(MQ^2\right).
\end{split}
\end{align}
From relations (\ref{bas1}), (\ref{bas.2}), (\ref{S_2.}), (\ref{1}) and  (\ref{2}) we obtain 
\begin{align}
\begin{split}
\label{S_2}
S_{2\underline{m}}
&\ll_k\left(Q^{-1/4}+p(2T)^{1/2+d(\alpha)}\left(\log p(2T)\right)^{2k}\right) \mathbf{K}\exp\left(MQ^2\right)\\
&\ll_k\left(Q^{-1/4}+p(2T)^{3/4+d(\alpha)}\right)\mathbf{K}\exp\left(MQ^2\right).
\end{split}
\end{align}
If now $\alpha\in\mathcal{A}_2$, then the second condition of relation \eqref{Q con} may not be satisfied.
 However, by the construction of the set $\mathcal{A}_2$ this can not happen too often. 
 Indeed, for every $\underline{m}\in\hat{M}\setminus\lbrace\underline{0}\rbrace$, the equation
$$\dfrac{Q_{\underline{m}}^+(\alpha)}{Q_{\underline{m}}^-(\alpha)}=\dfrac{x+\alpha}{y+\alpha}$$
has at most one solution in the positive integers, $\left(x_{\underline{m}},y_{\underline{m}}\right)$ say, with $x_{\underline{m}}\neq y_{\underline{m}}$, as follows from the irrationality of $\alpha$. 
In case such a solution does not exist in the set $\left(\mathbb{N}\cap[Q,+\infty)\right)^2$, the estimate for $S_{2\underline{m}}$ is the same as in \eqref{S_2}. 
If it exists, then
we have to add in \eqref{S_2} the term 
$$\dfrac{\left(\log\left(x_{\underline{m}}+\alpha\right)\right)^k\left(\log\left(y_{\underline{m}}+\alpha\right)\right)^k}{\left(x_{\underline{m}}+\alpha\right)^\sigma\left(y_{\underline{m}}+\alpha\right)^\sigma}T,$$
where $x_{\underline{m}}$ and $y_{\underline{m}}$ are both greater than $Q$ and at least one of them is greater than $\exp\left(2Q^2\right)$. Therefore, for sufficiently large $Q\gg_k1$, the additional term is bounded above by
$$\dfrac{T}{\exp\left(Q^2\right)Q^{1/2}}.$$
In view of the preceding and  (\ref{coef}), (\ref{series est.}), (\ref{S_1}) and (\ref{S_2}), we conclude that
\begin{align*}
{I}_1\ll_k&\left[\delta^Q+O\left(\dfrac{\delta^Q\exp\left(Q^2\right)}{M}\right)\right]\left(Q^{-1/4}T+p(2T)^{3/4}\right)+\\
&+\delta^Q\exp\left(Q^2\right)\left[\left(Q^{-1/4}
+p(2T)^{3/4+d(\alpha)}\right)\mathbf{K}\exp\left(MQ^2\right)+\dfrac{T}{\exp\left(Q^2\right)Q^{1/2}}\right]
\end{align*}
or
\begin{align}\label{I_1}
\begin{split}
{I}_1
\ll_k&\,\,Q^{-1/4}\left[2+\dfrac{\exp\left(Q^2\right)}{M}+\dfrac{\mathbf{K}\exp\left((M+1)Q^2\right)}{T}\right]\delta^QT+\\
&+\left[1+\dfrac{\exp\left(Q^2\right)}{M}+\mathbf{K}\exp\left((M+1)Q^2\right)\right] \delta^Qp(2T)^{3/4+d(\alpha)}.
\end{split}
\end{align}
Observe that
\begin{align*}
p(2T)^{3/4+d(\alpha)}\ll_{d,k}T^{\frac{3+4(d(\alpha)-d)}{4d}}{{T}}.
\end{align*} 
Then, for $
Q\gg_{k}1/\varepsilon^8,
$
$M\gg\exp\left(2Q^2\right)$,
$\alpha\in\mathcal\mathcal{A}(Q,M)$, $ d\geq d(\alpha)+1$ and 
\begin{align}\label{T}T\gg_{d,k}\left(\mathbf{K}\exp\left((M+2)Q^2\right)\right)^{\frac{4d}{4(d-d(\alpha))-3}},
\end{align} with suitable constants in $\gg$, we deduce from (\ref{bas2}) and (\ref{I_1}) that
\begin{align}\label{I1}
{I}_1<\dfrac{\varepsilon^2}{2}\int\limits_{T}^{2T}\Lambda_Q\left(\gamma_{Q}(\tau,\alpha)-\underline{\theta}_1\right)\mathrm{d}\tau
\end{align}
for every $\mathbf{A}\left(\left(d+1/(2d)\right)^{-1}\right)\leq\sigma\leq1$ and $\underline{\theta}_1\in\mathbb{R}^Q$.

Finally, we estimate ${I}_2$ by
\begin{align}\label{I2}
{I}_2\ll_{d,k}T^{-\nu}\int\limits_{T}^{2T}\Lambda_Q\left(\gamma_{Q}(\tau,\alpha)-\underline{\theta}_1\right)\mathrm{d}\tau.
\end{align}
The theorem now follows from  (\ref{T})-(\ref{I2}).
\end{proof}

\section{Proofs of Theorem \ref{weak tran} and Theorem \ref{weak.}}

\begin{proof}[Proof of Theorem \ref{weak tran}]
Let $\sigma$, $N$, $A$, $\varepsilon$, $\textbf{a}$, $R$ and $Q_0$ be as in
Lemma \ref{MAIN1}. Then, for every $Q\geq Q_0$ and $\alpha\in[A,1]$,
the system of inequalities
\begin{align*}
\left|{\left.\dfrac{\partial^k}{\partial s^k}\zeta_Q\left(s,\underline{\theta},\alpha\right)\right|}_{s=\sigma}-a_k\right|<\dfrac{\varepsilon}{4},\,\,\,\,\,\,\,\,\,k=0,\dots ,N,
\end{align*}
has a solution $\underline{\theta}_0=\underline{\theta}_0(\alpha)$.
If we take $\delta=Q^{-2}$,
then  the inequality
\begin{align}\label{delta}
|\theta_n-\theta_{0n}|\leq\delta
\end{align}
implies that
\begin{align*}
\left|{\left.\dfrac{\partial^k}{\partial s^k}\left(\zeta_Q\left(s,\underline{\theta},\alpha\right)-\zeta_Q\left(s,\underline{\theta}_0,\alpha\right)\right)\right|}_{s=\sigma}\right|
&\leq\sum\limits_{n=0}^{Q-1}\dfrac{\left(\log(n+\alpha)\right)^k\left|\mathrm{e}(\theta_n)-\mathrm{e}(\theta_{0n})\right|}{(n+\alpha)^{\sigma}}\tag*{}\\
&\ll\dfrac{1}{\alpha^{\sigma}}\delta Q\log^NQ\tag*{}\\
&\ll A^{-1/2}Q^{-1}\log^NQ\tag*{}\\
&\ll_{N,A}Q^{-1/2}
\end{align*}
for $k=0,\dots,N$.
Thus, the system of inequalities
\begin{align}\label{syst}
\left|{\left.\dfrac{\partial^k}{\partial s^k}\zeta_Q\left(s,\underline{\theta},\alpha\right)\right|}_{s=\sigma}-a_k\right|<\frac{\varepsilon}{2}\,{{<\left(2\dfrac{Q^{2}+1}{Q^{2}-1}\right)^{1/2}\dfrac{\varepsilon}{2}}},\,\,\,\,\,\,\,\,\,k=0,\dots, N,
\end{align}
is satisfied whenever $\alpha\in[A,1]$, $Q\gg_{N,A}Q_0+1/\varepsilon^4$ and (\ref{delta}) holds.
On the other hand Lemma \ref{MAIN3} yields, for every $Q\geq C_3(N)/\varepsilon^8$, $M\geq C_4\exp\left(2Q^2\right)$, $\alpha\in\mathcal{A}(Q,M)\cap[A,1]$ and $d\geq d(\alpha)+1$, the existence of a positive number $\nu(d,N)$ such that, for every
\begin{align*}
{{T\geq C_5(d,N)\max\left\{\left(\mathbf{K}\exp\left((M+2)\exp\left(Q^2\right)\right)\right)^{\frac{4d}{4(d-d(\alpha))-3}},\varepsilon^{-2\nu}\right\}}},
\end{align*}
we have
{{\begin{align}\label{L_q}
\int\limits_{T}^{2T}\Lambda_Q(\gamma_{Q}(\tau,\alpha)-\underline{\theta}_0)\mathrm{d}\tau
\geq \delta^Q\left(1-Q^{-2}\right)
T
\end{align} }}
and
\begin{align}\label{last}
\begin{split}
&\sum\limits_{k=0}^N\int\limits_{T}^{{2T}}\Lambda_Q\left(\gamma_{Q}(\tau,\alpha)-\underline{\theta}_0\right)\left|\zeta^{(k)}\left(\sigma+i\tau;\alpha\right)-\left.\dfrac{\partial^k}{\partial s^k}\zeta_Q\left(s+i\tau,\underline{0},\alpha\right)\right|_{s=\sigma}\right|^2\mathrm{d}\tau\\
&< \sum\limits_{k=0}^N\dfrac{\varepsilon^2}{4(N+1)}\int\limits_{T}^{{2T}}\Lambda_Q(\gamma_{Q}(\tau,\alpha)-\underline{\theta}_0)\mathrm{d}\tau\\
&=\dfrac{\varepsilon^2}{4}\int\limits_{T}^{{2T}}\Lambda_Q(\gamma_{Q}(\tau,\alpha)-\underline{\theta}_0)\mathrm{d}\tau\\
&<\frac{\varepsilon^2}{4}\delta^Q\left(1+Q^{-2}\right)T
\end{split}
\end{align}
for $\mathbf{A}(1/(d+1/(2d)))\leq\sigma\leq1$.

Let $Q\gg_{N,A}\left(Q_0+C_3(N)/\varepsilon^8\right)$, $\mathbf{A}(1/(d+1/(2d)))\leq\sigma\leq1$ and assume that there is no solution $\tau$ in $\left[T,2T\right]$ for the system of inequalities \eqref{SYST}.
Then, for every $\tau\in[T,2T]$,  there is a $k_\tau\in\left\{0,\dots,N\right\}$ such that
\begin{align*}
\sum\limits_{k=0}^N&\left|\zeta^{(k)}\left(\sigma+i\tau;\alpha\right)-\left.\dfrac{\partial^k}{\partial s^k}\zeta_Q\left(s+i\tau,\underline{0},\alpha\right)\right|_{s=\sigma}\right|^2\\
&\geq\left|\zeta^{(k_\tau)}\left(\sigma+i\tau;\alpha\right)-\left.\dfrac{\partial^{k_\tau}}{\partial s^{k_\tau}}\zeta_Q\left(s+i\tau,\underline{0},\alpha\right)\right|_{s=\sigma}\right|^2\\
&\geq\dfrac{1}{2}\left|\zeta^{(k_\tau)}\left(\sigma+i\tau;\alpha\right)-a_{k_\tau}\right|^2-\left|a_{k_\tau}-\left.\dfrac{\partial^{k_\tau}}{\partial s^{k_\tau}}\zeta_Q\left(s+i\tau,\underline{0},\alpha\right)\right|_{s=\sigma}\right|^2\\
&\geq\dfrac{\varepsilon^2}{2}-\dfrac{\varepsilon^2}{4}\\
&=\dfrac{\varepsilon^2}{4},
\end{align*}
as follows from \eqref{syst}.
However, this contradicts (\ref{last}). 

Now let \begin{align}\label{D}
\mathcal{U}_T(\alpha):=\left\{\tau\in\left[T,2T\right]:\Lambda_Q(\gamma_{Q}(\tau,\alpha)-\underline{\theta}_0)\neq0\right\}.
\end{align}
By definition $\Lambda_Q$ is bounded above by 1. 
This and (\ref{L_q}) imply that 
\begin{align}\label{meas1}
\mathrm{m}\left(\mathcal{U}_T(\alpha)\right)\geq\delta^Q\left(1-Q^{-2}\right)
T.
\end{align}
If $\mathcal{M}_T(\alpha,\sigma)$ is the set of those $\tau\in \mathcal{U}_T(\alpha)$ for which the system of inequalities
\begin{align*}
\left|\zeta^{(k)}\left(\sigma+i\tau;\alpha\right)-a_k\right|<\left(2\dfrac{Q^2+1}{Q^2-1}\right)^{1/2}\varepsilon,\,\,\,\,\,\,\,\,\,k=0,\dots, N,
\end{align*}
is satisfied, then relations (\ref{syst})-(\ref{meas1}) yield that
\begin{align*}
\mathrm{m}\left(\mathcal{M}_T(\alpha,\sigma)\right)\geq\dfrac{1}{2}\delta^Q\left(1-Q^{-2}\right)
T.
\end{align*}
For if that was not true, we would have
\begin{align*}
\sum\limits_{k=0}^N\left|\zeta^{(k)}\left(\sigma+i\tau;\alpha\right)-\left.\dfrac{\partial^k}{\partial s^k}\zeta_Q\left(s+i\tau,\underline{0},\alpha\right)\right|_{s=\sigma}\right|^2\geq\dfrac{\varepsilon^2}{2}\dfrac{Q^2+1}{Q^2-1}
\end{align*}
for every $\tau$ in the set of positive measure $\mathcal{U}_T(\alpha)\setminus \mathcal{M}_T(\alpha,\sigma)$.
 It would then follow from (\ref{syst}), (\ref{L_q}) and \eqref{meas1} that
\begin{align*}
&\int\limits_{T}^{2T}\Lambda_Q\left(\gamma_{Q}(\tau,\alpha)-\underline{\theta}_0\right)\sum\limits_{k=0}^N\left|\zeta^{(k)}\left(\sigma+i\tau;\alpha\right)-\left.\dfrac{\partial^k}{\partial s^k}\zeta_Q\left(s+i\tau,\underline{0},\alpha\right)\right|_{s=\sigma}\right|^2\mathrm{d}\tau\tag*{}\\
&\geq \frac{\varepsilon^2}{2}\dfrac{Q^2+1}{Q^2-1}\int_{\mathcal{U}_T(\alpha)\setminus \mathcal{M}_T(\alpha,\sigma)}\Lambda_Q\left(\gamma_{Q}(\tau,\alpha)-\underline{\theta}_0\right)\mathrm{d}\tau\\
&\geq\frac{\varepsilon^2}{2}\dfrac{Q^2+1}{Q^2-1}\left[\int_{\mathcal{U}_T(\alpha)}\Lambda_Q\left(\gamma_{Q}(\tau,\alpha)-\underline{\theta}_0\right)\mathrm{d}\tau-\mathrm{m}\left(\mathcal{M}_T(\alpha,\sigma)\right)\right]\\
&>\frac{\varepsilon^2}{4}\delta^Q\left(1+Q^{-2}\right)T,
\end{align*}
which contradicts (\ref{last}).
\end{proof}

\begin{proof}[Proof of Theorem \ref{weak.}]
Beginning with the Taylor series of $f$, 
$$
f(s)=\sum_{k=0}^\infty\frac{f^{(k)}(s_0)}{k!}(s-s_0)^k,
$$
valid for $s\in \mathcal{K}$, we observe, by Cauchy's formula
$$
f^{(k)}(s_0)={\frac{k!}{ 2\pi i}}\int_{\vert s-s_0\vert=r}{\frac{f(s)}{(s-s_0)^k}}\mathrm{d} s,
$$
that $\left\vert f^{(k)}(s_0)\right\vert \leq k!Mr^{-k}$, where $M:=\max_{\vert s-s_0\vert=r}\vert f(s)$. Fixing a number $\delta_0\in(0,1)$, we get
$$
\left\vert{\frac{f^{(k)}(s_0)}{ k!}}(s-s_0)^k\right\vert\leq M\delta_0^k
$$
for $\vert s-s_0\vert\leq \delta_0r$. 
If $\varepsilon\in(0,\vert f(s_0)\vert)$, we can find $N=N(\delta_0,\varepsilon,M)$ such that  
$$
\Sigma_1:=\left\vert f(s)-\sum_{k=0}^N{\frac{f^{(k)}(s_0)}{k!}}(s-s_0)^k\right\vert<\varepsilon,
$$
for $\vert s-s_0\vert\leq \delta_0r$.

Now let $\delta\in(0,\delta_0)$. 
Then, of course, the latter inequality holds in particular for $s$ satisfying $\vert s-s_0\vert\leq \delta r$.
 Now we apply Theorem \ref{weak tran} with $a_k=f^{(k)}(s_0)$,
 $k=0,\dots, N$.
 Then, for $\alpha\in\mathcal{A}(Q,M)\cap[A,1]$ of degree at most $d_0-1$,
and $T$ satisfying relation \eqref{BBB}, there exists $t_1\in[T,2T]$ such that
$$
\vert \zeta^{(k)}(\sigma_0+it_1;\alpha)-f^{(k)}(s_0)\vert<\varepsilon,\,\,\,k=0,\dots,N.
$$
Thus, 
\begin{eqnarray*}
\Sigma_2&:=&\left\vert \sum_{k=0}^{N}{\frac{\zeta^{(k)}(\sigma_0+it_1;\alpha)}{ k!}}(s-s_0)^k-\sum_{k=0}^N{\frac{f^{(k)}(s_0)}{ k!}}(s-s_0)^k\right\vert\\
&<& \varepsilon \sum_{k=0}^N{\frac{(\delta r)^k}{k!}}\\
&<&\varepsilon \exp(\delta r),
\end{eqnarray*}
for $\vert s-s_0\vert\leq \delta_0r$.
Now write $\tau=t_1-t_0$, then $1+it_1=s_0+i\tau$. 

Next we use the Taylor expansion for $\zeta(s;\alpha)$ on the shifted disk $\mathcal{K}+i\tau$.
 For this purpose we need to exclude the simple pole at $s=1$; 
 hence we also request $T>r$. 
 Under this assumption we have
$$
\zeta(s+i\tau;\alpha)=\sum\limits_{k=0}^{\infty}{\frac{\zeta^{(k)}(s_0+i\tau;\alpha)}{k!}}(s-s_0)^k
$$
for $s\in \mathcal{K}$. Let $M(\tau):=\max_{\vert s-s_0\vert=r}\vert \zeta(s+i\tau;\alpha)\vert$. Then, again by Cauchy's formula, 
$$
\left\vert {\frac{\zeta^{(k)}(s_0+i\tau;\alpha)}{k!}}(s-s_0)k\right\vert\leq M(\tau)\delta^k
$$
for $\vert s-s_0\vert\leq \delta_0r$.
Hence, 
\begin{eqnarray*}
\Sigma_3&:=&  \left\vert \zeta(s+i\tau;\alpha)-\sum_{k=0}^N{\frac{\zeta^{(k)}(s_0+i\tau;\alpha)}{k!}}(s-s_0)^k\right\vert\\
&=&\left\vert \sum_{k>N}{\frac{\zeta^{(k)}(s_0+i\tau;\alpha)}{k!}}(s-s_0)^k\right\vert  \\
&\leq & M(\tau){\frac{\delta^N} {1-\delta}},
\end{eqnarray*}
for $\vert s-s_0\vert\leq \delta_0r$.
In combination with the above estimates this yields
\begin{eqnarray*}
\vert \zeta(s+i\tau;\alpha)-f(s)\vert &\leq & \Sigma_1+\Sigma_2+\Sigma_3< \varepsilon +\varepsilon \exp(\delta r)+M(\tau){\frac{\delta^N}{ 1-\delta}} 
\end{eqnarray*}
for $\vert s-s_0\vert\leq \delta r$. Now we choose $\delta>0$ such that $M(\tau){\frac{\delta^N}{1-\delta}}=\varepsilon(2-\exp(\delta r))$; this choice is possible since the left hand side tends to zero as $\delta\to 0$ while the right hand side tends to $\varepsilon>0$, resp. the left hand side tends to infinity but the right hand side remains bounded as $\delta\to 1$. This proves the theorem.  
\end{proof}

\section{Concluding Remarks}

We begin with a historical note. 
Hurwitz \cite{hurwitz1882einige} himself treated only Hurwitz zeta-functions with a rational parameter.
 In his investigations on Dirichlet's analytic class number formula, he studied Dirichlet series of the form 
$$
\sum_{n\equiv a\bmod\, m}n^{-s}
$$
which can be rewritten as $m^{-s}\zeta(s;{\frac{a}{ m}})$. 
It appears that switching from a rational to an irrational parameter does not affect analytic continuation, functional identities and the order of growth, however, the zero-distribution definitely depends and the general distribution of values might depend  on the diophantine nature of the parameter. 
Further generalizations of the Riemann zeta-function and Dirichlet $L$-functions were in Hurwitz's time also studied by Lerch and Lipschitz.
 For details we refer to the monograph \cite{zbMATH01924169} by Laurin\v cikas and Garunk\v stis on the Lerch zeta-function (which covers the case of Hurwitz's zeta). 

Our next remark shall classify the various proofs of universality properties for the Hurwitz zeta-function $\zeta(s;\alpha)$ in the literature so far, namely the results of Bagchi \cite{bagchi1981statistical} and Gonek \cite{gonek1979analytic}. 
For rational $\alpha$ there is a representation of $\zeta(s;\alpha)$ in terms of Dirichlet $L$-functions with pairwise inequivalent characters which allows to apply a joint universality theorem for those due to Voronin \cite{voronin1975functional} in order to deduce the desired approximation property; for $\alpha\neq {\frac{1}{ 2}},1$, the target function may even vanish (whereas for $\alpha={\frac{1}{ 2}}$ or $1$ the Hurwitz zeta-function is essentially equal to a Dirichlet $L$-function and Riemann's $\zeta$, respectively, and too many zeros off the critical line would contradict classical density theorems). 
If the parameter $\alpha$ is transcendental, one can mimic Voronin's proof using the linear independence of the numbers $\log(n+\alpha)$ for non-negative integers $n$ (in place of the logarithms of the prime numbers in case of $\zeta$).
 In all these results the approximating shifts form a set of positive lower density (as in the universality theorem for the Riemann zeta-function).   
 
 Although our results are far from being satisfactory in comparison with the case of transcendental parameter, they shed light to another major topic in universality theory, that of effective lower bounds for the lower density of the set of approximating shifts as well as estimating explicitly $T$ such that an approximating shift $\tau$ lies in $[T,2T]$. 
 It is evident from the proof of Theorem \ref{MAIN3} that if we have an estimate of the form
 \begin{align}\label{Tran}
 |P(\alpha)|\geq D H^C,
 \end{align}
 for transcendental $\alpha$ and integral polynomial $P$, where $D$ and $C$ may depend on the degree of $P$ but not in its height $H$, then we can obtain the same lower bound for $$\liminf\limits_{T\to\infty}\dfrac{1}{T}\mathcal{M}_T(\alpha,\sigma)$$ as in Theorem \ref{weak tran}, whenever $\sigma$ is sufficiently close to $1$. If in addition $D$ and $C$ are effectively computable, then we can estimate explicitly a $T$ such that $\tau\in[T,2T]$. 
 In direction of \eqref{Tran} we refer to the monograph of Bugeaud \cite{bugeaud2004approximation}, where the classification of transcendental numbers into $S$-, $T$- and $U$-numbers is given in detail.
\bibliography{bib1}{}
\bibliographystyle{siam-new}
\small
Athanasios Sourmelidis, J\"orn Steuding\\
Institute for Mathematics, W\" urzburg University, \\
Emil-Fischer Str. 40, 97074 W\"urzburg, Germany\\
 athanasios.sourmelidis@mathematik.uni-wuerzburg.de\\
steuding@mathematik.uni-wuerzburg.de
\end{document}